\documentclass[12pt, a4 paper]{amsart} 

\usepackage[psamsfonts]{amssymb} 
\usepackage{amsfonts,amsmath,mathabx}
\usepackage{color}

\usepackage{amsthm, amssymb}
\usepackage{amsfonts} 
\usepackage{epsfig,multicol} 

\usepackage{xypic}
\usepackage{hyperref}
\usepackage{psfrag}

\input xy
\xyoption{all}

\makeatletter 
\let\@wraptoccontribs\wraptoccontribs
\makeatother

\newtheorem{theorem}{Theorem}[section]

\newtheorem{proposition}[theorem]{Proposition}
\newtheorem{lemma}[theorem]{Lemma}
\newtheorem{corollary}[theorem] {Corollary}

\newtheorem{problem}[theorem]{Problem}

\theoremstyle{remark}
\newtheorem{remark}[theorem]{Remark}
\newtheorem{remarks}[theorem]{Remarks}

\theoremstyle{definition}
\newtheorem{definition}[theorem]{Definition}

\def\G{\Gamma}
\def\ssm{\smallsetminus}

\def\-{\overline}
  
\def\id{{\rm{id}}}
\def\dom{{\rm{dom}}}
\def\ran{{\rm{ran}}}

\def\r{\rho}

\def\P{\mathcal{P}}
\def\AR{\<A\mid R\>}

\def\pr{\mathcal{B}_\rho(\P)}
\def\prprime{\mathcal{B}_\rho(\P')}

\def\H{\mathcal{H}}

\def\T{\mathcal{T}}

\def\P{\mathcal{P}}

\def\G{\Gamma}

\def\<{\langle}
\def\>{\rangle}

\newcommand{\inv}{^{-1}}
\newcommand{\p}{\varphi}
\newcommand{\pinv}{{\p \inv}}
\newcommand{\til}[1]{\ensuremath{\widetilde {#1}}}

\begin{document}

\title[Permutoids, pseudogroups, and undecidability]{Undecidability and the developability of permutoids and rigid pseudogroups}
\contrib[(with an appendix by]{Benjamin Steinberg)}

\author[Martin R. Bridson]{Martin R.~Bridson}
\address{Mathematical Insitute,
Andrew Wiles Building, Woodstock Road,
Oxford,
OX2 6GG,
U.K. }
\email{bridson@maths.ox.ac.uk}
 
 \author[Henry Wilton]{Henry Wilton}
\address{Centre for Mathematical Sciences, Wilberforce Road, Cambridge, CB3 0WB, U.K.}
\email{h.wilton@maths.cam.ac.uk}

\date{25 November 2016}
 
\thanks{Bridson and Wilton both thank the EPSRC for its financial support. Bridson's work is also supported by  a Wolfson Research Merit Award
from   the Royal Society; he thanks them and   he thanks
the Erwin Schr\"odinger International Institute for Mathematical Physics in Vienna for its hospitality during the writing of this article.  }


\begin{abstract}
A \emph{permutoid} is a set of partial permutations that contains the identity and is such that partial compositions, when defined, have at most one extension in the set.   
In 2004 Peter Cameron conjectured that there  can exist no algorithm that determines whether or not a permutoid based on a finite set can be completed to a finite permutation group.
In this note we prove Cameron's conjecture by relating it to our recent work on the profinite
triviality problem for finitely presented groups.  We also prove that the existence problem for finite developments of rigid pseudogroups is unsolvable.  In an appendix, Steinberg recasts these results in terms of inverse semigroups.
\end{abstract}

\maketitle

\def\G{\Gamma} 

\section{Introduction} Across many contexts in mathematics one encounters extension problems
of the following sort: given a set $S$ of partially-defined automorphisms of an object $X$,
one seeks an object $Y\supset X$ and a set of automorphisms $\hat S$  
of $Y$ such that each $s\in S$ has an extension $\hat s\in\hat S$.  
In the category of finite sets, this problem is trivial because any partial permutation of a set can be extended to a permutation of that set. Less trivially,
Hrushovski \cite{ehud} showed
that extensions always exist in the category of finite graphs.
But if one requires extensions to respect (partially defined) compositions in $S$, such existence problems become more subtle.  
In 2004
Peter Cameron \cite{cameron_extending_2004} conjectured that there does not exist an algorithm that can solve the following extension problem.

\begin{problem}\label{prob}  
{\bf{Given}} partial permutations $p_1, ..., p_m$ of a finite set $X$ (that is, bijections between subsets
of $X$) such that
\begin{enumerate}
\item
$p_1={\rm{id}}_X$, and
\item
for all $i,j$ with $\dom(p_i)\cap \ran(p_j)\neq \emptyset$,
there is at most one $k$ such that $p_k$ extends $p_i\cdot p_j$
\end{enumerate}
{\bf{decide}} whether or not there exists a finite set $Y$ containing $X$, and permutations $f_i$ of $Y$ extending 
$p_i$ for $i=1,...,m$, such that if $p_k$ extends $p_i\cdot p_j$ then $f_i \circ f_j = f_k$.
\end{problem}

We shall prove that this problem is indeed algorithmically unsolvable by relating it to our
recent work on the triviality problem for finitely presented profinite groups \cite{BWprotrivial}. In order to achieve this, we develop some formalism: a
collection of partial permutations as in Problem \ref{prob} is called
a {\em{permutoid}}; in Section \ref{s:defns}
we define morphisms, quotients, and developments of permutoids.
In this terminology, Cameron's conjecture is that there does not
exist an algorithm that can decide whether or not a finite permutoid is developable.
Cameron \cite{cameron_extending_2004} associated a permutoid to a finite group
presentation (cf.~Proposition \ref{p:half}) and observed that if the group has no finite quotients then the permutoid is not developable.  If the converse were to hold,
Cameron's conjecture would follow easily from the constructions in 
\cite{BWprotrivial}, but unfortunately it does not (see Remark \ref{r:Cam}). It is to obviate this difficulty that we introduce quotient permutoids.

In the final section of this paper we shall explain how our main construction
also can be adapted to prove a similar undecidability result for rigid pseudogroups.

Lastly, in an appendix, Benjamin Steinberg explains how the results of this paper can be recast in the language of inverse semigroups.

\section{Partial Permutations and Permutoids}\label{s:defns}


A {\em partial permutation} of a set $X$ is a bijection between two 
non-empty subsets of $X$. We denote the
domain and range of a partial permutation $p$  by $\dom(p)$ and $\ran(p)$ respectively.
By definition, $q$ {\em{extends}} $p$ if $\dom(p) \subset \dom(q)$  and $q(x) = p(x)$ for all
$x\in\dom(p)$. 
The composition $p\cdot q$ of two partial permutations $p,q$ on $X$ is 
defined if $\ran(p)\cap\dom(q)$ is non-empty: 
$p\cdot q (x) = p(q(x))$ for $x\in q^{-1}(\ran(q)\cap\dom(p))$. 

\begin{definition}\label{defP} A {\em{permutoid}} $(\Pi; X)$
is a set $\Pi$ 
of partial permutations of a set $X$ such that
\begin{enumerate}
\item $\Pi$ contains $1_X$, the identity map of $X$;
\item  for all $p,q\in\Pi$ there exists at most one $r\in\Pi$ such that $r$ extends $p\cdot q$ (if the composition exists).
\end{enumerate}
The permutoid is {\em{finite}} if $X$ is finite, and \emph{trivial} if $\Pi=\{1_X\}$.

A {\em{morphism}} of permutoids $(\Pi; X) \overset{(\phi,\Phi)}\to (\Pi'; X')$ is a pair of maps 
$\phi:\Pi\to\Pi'$ and $\Phi:X\to X'$ so that:
\begin{enumerate}
\item $\phi(1_X)=1_{X'}$;
\item $\Phi(\dom(p))\subseteq \dom(\phi(p))$
and $\phi(p)(\Phi(x)) = \Phi(p(x))$  for all $p\in \Pi$ and $x\in\dom(p)$;
\item if $r$ extends $p\cdot q$, with $p,q,r\in\Pi$, then $\phi(r)$ extends $\phi(p)\cdot\phi(q)$.
\end{enumerate}

The morphism $(\phi,\Phi)$ is an {\em{isomorphism}} if $\Phi$ and $\phi$ are bijections
and $\phi(p) = \Phi\circ p \circ \Phi^{-1}$ for all $p\in\Pi$.

The morphism $(\phi,\Phi)$ is a {\em{quotient map}} if $\Phi$ and $\phi$ are surjections.

The morphism $(\phi,\Phi)$ is an {\em{extension}} if $\Phi$ is injective.

An extension $(\Pi; X) \overset{(\phi,\Phi)}\to (\Pi'; X')$ is {\em{complete}} if $\Pi'\subseteq {\rm{Perm}}(X')$; in other words $\dom(p')=\ran(p')=X'$ for all $p'\in\Pi'$.

If a finite permutoid $(\Pi; X)$ admits a finite complete extension, then 
$(\Pi; X)$ is said to be {\em{developable}} and $(\Pi'; X')$ is called a
development.
\end{definition}

\begin{remarks} 
(1) Cameron's Conjecture asserts that 
there is no algorithm that can determine the developability 
of a finite permutoid.
\smallskip

(2) If $(\Pi; X) \overset{(\phi,\Phi)}\to (\Pi'; X')$ is an extension, then
 $\phi$ will fail to be injective precisely when $\Pi$ contains two distinct
restrictions of some $p\in\Pi'$. For example, if $\Phi$ is the identity
map on $X$ and $p$ is a permutation with at least two points $x_1,x_2$ in its support,
then we obtain an extension  $(\Pi; X) \overset{(\phi,\id)}\to (\Pi'; X)$ 
 by defining $p_i=p|_{x_i},\ \Pi=\{\id, p_1, p_2\},\ \Pi'=\{\id, p\}$ and $\phi(p_i)=p$.
\end{remarks}

\begin{definition} The {\em{universal group}} of a permutoid $(\Pi;X)$ is 
$$
\Gamma(\Pi;X) := \< \Pi \mid pq=r \text{ if $r$ extends $p\cdot q$ }\>.
$$
\end{definition}

\begin{lemma}\label{l:quot}
\begin{enumerate}
\item If $(\Pi; X) \overset{(\phi,\Phi)}\to (\Pi'; X')$ is a morphism, then
$p\mapsto \phi(p)$ defines a homomorphism of groups
$$
\phi_*  : \Gamma(\Pi;X)\to \Gamma(\Pi';X'),
$$
and if $(\phi,\Phi)$ is a quotient morphism then $\phi_*$ is surjective.
\item
If $\Pi\subseteq{\rm{Perm}}(X)$, then there is an epimorphism
$ \Gamma(\Pi;X)\to \<\Pi\> \le {\rm{Perm}}(X)$.
\item 
If a non-trivial finite permutoid $(\Pi; X)$ is developable,
 then $\Gamma(\Pi;X)$ has a non-trivial finite quotient.
\end{enumerate}
\end{lemma}

\begin{proof} Parts (1) and (2) are immediate from the definitions. For (3), let $(\Pi; X) \overset{(\phi,\Phi)}\to (\Pi'; X')$ be a finite complete extension.  Note that if $p\in \Pi$ is not $\mathrm{id}_X$ then $p(x)\neq x$ for some $x\in X$, and hence $\phi(p)\neq \mathrm{id}_{X'}$. It follows that the image of $\phi_*$ is non-trivial, and so (3) follows from (2).
\end{proof}

\section{Cameron permutoids}

A {\em marked group} is a pair $(G,A)$ where $G$ 
is a group and $A$ is a generating set. Let $\rho$ be a positive integer. 
Let $B_\r \subset G$ be the set of elements that can be expressed as a word of length 
at most $\r$ in the generators and their inverses, and define $B_{2\r}$ similarly.
Define $p_1$ to be the identity map on $B_{2\r}$, and
for each $b\in B_\r\ssm\{1\}$ define  $p_b: B_\r\to B_{2\r}$ to be the restriction of left multiplication by $b$; that
is, $p_b(x) = bx$. Let  $\Pi_\r=\{p_b \mid b\in B_\r\}$.

\begin{lemma}\label{l:BisP} 
The pair $(\Pi_\r; B_{2\r})$ is a permutoid. If $A$ is finite then this permutoid is finite.
\end{lemma}

\begin{proof} Each element $g\in G$ is uniquely determined by its action by left multiplication
on any point $x\in G$. Thus, for all $b,b'\in  B_\r$,
if $bb'$ lies in $ B_\r$ then $p_{bb'}$ is the unique element of $\Pi_\r$ extending $p_b\cdot p_{b'}$,
and if not then no element of $\Pi_\r$ extends $p_b\cdot p_{b'}$.
\end{proof}

\begin{definition}\label{d:B}
Given a marked group $(G,A)$  and a positive integer $\rho$, $(\Pi_\r; B_{2\r})$ is called\footnote{in recognition of the fact that Peter Cameron considered these objects in \cite{cameron_extending_2004}}
 a \emph{Cameron permutoid}.  If $\mathcal{P}\equiv \langle A\mid R \rangle$ is a finite presentation for a group $G$, then we write $\pr$ to denote the Cameron permutoid $(\Pi_\r; B_{2\r})$.
\end{definition}

\begin{remark}\label{algo}
It is important to note that, in order to construct $\pr$ from a finite presentation
$\P$, one needs to be able to 
calculate which words of length at most $\rho$ in the generators represent equal elements of the group,
and for each pair of such elements $b,x$, one needs to calculate $bx$. This can be achieved
in an algorithmic manner {\em{provided}} that one has a solution to the word problem in $|\P|$.
And in order to achieve the construction for all $\rho>0$ and all presentations in a class $\mathfrak{P}$,
one needs a {\em{uniform solution}} to the word problem for the groups in  $\mathfrak{P}$.
\end{remark}

We have arranged the definitions so that the following lemma is obvious.

\begin{lemma}\label{l:onto}
For all presentations $\P \equiv \AR$ and $\P'\equiv \<A\mid R'\>$ with $R\subset R'$, the natural epimorphism $|\P|\overset{\pi}\to |\P'|$ induces a quotient map of permutoids $\pr\overset{(\phi, \Phi)}\to\prprime$, where $\Phi$ is the restriction of $\pi$ to $B_{2\r}$ and $\phi(p_b) = p_{\pi(b)}$.
\end{lemma}

If $\rho$ is large enough then there is a natural isomorphism $\Gamma(\pr)\cong |\P|$.  In order to see this, we need the following well-known triangulation procedure.

\begin{lemma}\label{l:T} Let $\P\equiv\AR$ be a finite presentation, let $m$ be an integer greater than half the length
of the longest relation in $R$, let $B$ be the set of elements of $G=|\P|$ that can be expressed as
words of length at most $m$ in the free group $F(A)$,  let $T$ be the set of words $w\in F(B)$ of length three
that equal the identity in $G$, and let $\mathcal T \equiv \<B\mid T\>$.
Then, the natural map $A\to B\subset G$ induces an isomorphism $|\P|\to |\mathcal{T}|$.
\end{lemma}

\begin{proposition}\label{p:half} If $\rho\geq 1$, then there is a natural epimorphism of groups $\Gamma(\pr)\to |\P|$, and if
$\rho$ is greater than half the length of the longest relator in $R$, this is an isomorphism.
\end{proposition}

\begin{proof}
By definition,
$$
\Gamma(\pr) = \< p_b\ (b\in B_\r) \mid p_{b_1}p_{b_2}=p_{b_3} \text{ if $b_1b_2=b_3$ in $|\P|$}\>,
$$
where $B_\r$ is the ball of radius $\r$ about the identity in $|\P|$ (with word metric $d_A$).  The homomorphism $\G(\pr)\to |\P|$ defined by $p_b\mapsto b$ is onto (since $\r\ge 1$ and the image of $B_1$ generates $|\P|$).  And if $\r$ is greater than half the length of the longest relator in $R$,  then modulo an obvious change of notation this is the isomorphism of Lemma \ref{l:T}. 
\end{proof} 

We need one final fact.

\begin{lemma} For all marked groups $(G,A)$ and all positive integers $\r>\r'>0$,
there is an extension of permutoids $(\Pi_{\r'}; B_{2\r'})\to (\Pi_\r; B_{2\r})$ given by
the inclusion $B_{2\r'}\hookrightarrow B_{2\r}$ and the map $\Pi_{\r'}\to \Pi_{\r}$
that extends left-multiplication from $B_{\r'}$ to $B_{\r}$.
\end{lemma}


\begin{corollary}\label{c:finite}
If $\P$ is a finite presentation of a finite group $G$ then, for all positive
integers $\rho$, the permutoid $\pr$ is developable.
\end{corollary}

\begin{proof}
If $\rho$ is sufficiently large then $B_\r=B_{2\r}=G$ and $\Pi_\rho < {\rm{Perm}}(G)$ is the subgroup 
consisting of left multiplications.
\end{proof}

\begin{remark} 
A permutoid defines a {\em{pree}} in an obvious manner. By definition,
a {\em{pree}} is a non-empty set $P$ with a partially defined binary operation, i.e.\ a subset $D\subseteq P\times P$ and a map $m:D\to P$. 
This terminology is due to Stallings \cite{stall_pree} (also \cite{rimlinger}); Baer \cite{baer} had earlier used the term {\em{add}} to describe such objects. Both Baer and Stallings established criteria that guarantee a pree will embed in the associated group
$$
G(P,m) := \< P \mid pq=m(p,q) \text{ for all } (p,q)\in D\>.
$$
\end{remark}

\section{Finite completions and finite quotients}

In the language of permutoids, Cameron's Conjecture (Problem \ref{prob}) is that
{\em developability is an undecidable property.}

\begin{theorem}\label{thm} There does not exist an algorithm that, given a
finite permutoid $(\Pi;X)$, can determine whether or not $(\Pi;X)$ is developable.
\end{theorem}

\begin{remark} It is clear that the isomorphism classes of finite permutoids form a recursive set, and
a naive search will eventually find a complete extension of a finite permutoid if such exists. The
content of the above theorem, then, is that there is no algorithm that
can enumerate the isomorphism classes of finite permutoids that do not
have a complete finite extension.
\end{remark}

In \cite[Theorem B]{BWprotrivial} we constructed a recursive set of finite presentations for biautomatic groups such that there is no algorithm that can determine which of these groups has a non-trivial finite quotient. The class of (bi)automatic groups admits a uniform solution to the word problem \cite[pp.~32, 112]{epstein_word_1992}. Theorem \ref{thm} therefore follows immediately from \cite[Theorem B]{BWprotrivial} and the next proposition.

\begin{proposition}\label{prop} Let $\mathfrak{P}$ be a class of finite presentations for groups drawn from
a class in which there is a uniform solution to the word problem. If there were an algorithm 
that could determine whether or not a finite permutoid was developable, then
there would be an algorithm that, given any presentation
$\P\in\mathfrak{P}$, could determine whether or not the group $|\P|$ had a non-trivial finite quotient.
\end{proposition}
\begin{proof}
Given $\P\in\mathfrak{P}$, take $\rho$ to be at least half the length of the longest relator and use the solution to the word problem to construct $\pr$ (cf.\ Remark \ref{algo}). Then list representatives $P_i$ for the finitely many isomorphism classes of the non-trivial
quotient permutoids.  The proposition now follows from the claim that $|\P|$ has a non-trivial finite quotient if and only if one of the $P_i$ is
developable.

On the one hand, if one of the  $P_i$ is developable then $\Gamma(P_i)$ has a finite quotient, by Lemma \ref{l:quot}(3), and hence, by Lemma \ref{l:quot}(1),  so does $\G(\pr)$, which, by Proposition \ref{p:half}, is isomorphic to  $|\P|$.  Conversely, if $|\P|$ has a non-trivial finite quotient, with presentation $\mathcal{P}'=\<A \mid R'\>$ say, where $R\subset R'$, then $\prprime$ will be a quotient permutoid of $\pr$, and the  $P_i$
isomorphic to it will be developable, by Corollary \ref{c:finite}. 
\end{proof}

\begin{remark}\label{r:Cam}
The key observation that if $\pr$ has a complete finite extension then $|\P|$ has a non-trivial finite quotient is due to Peter Cameron \cite{cameron_extending_2004}.  Note, however, that the converse to this observation does not hold: in general $\pr$ need not inject into any finite quotient of $|\P|$, even if such quotients exist.  For instance, $\P$ may present a non-cyclic group whose finite quotients are all cyclic, such as the example of Baumslag \cite{baumslag_non-cyclic_1969}.
\end{remark}

\section{Rigid developments and pseudogroups}\label{sec: Rigid development}

{\em Pseudogroups} play an important role in many geometric contexts. 
A pseudogroup of local homeomorphisms on a topological space $X$ is a collection $\H$ 
of homeomorphisms $h : U \to V$ of open sets of $X$ such that:
\begin{enumerate}
\item if $h : U \to V$ and $h':U'\to V'$ are in $\H$ then $h^{-1}$ and the composition
$h'h: h^{-1}(V\cap U')\to h'(V\cap U')$ belong to $\H$;
\item the restriction of $h$ to any open subset of $U$ belongs to $\H$;
\item $\id_X \in\H$;
\item  if a homeomorphism between open subsets of $X$ is the union of elements
of $\H$, then it too belongs to $\H$.
\end{enumerate}
We shall concentrate on the case where $X$ is a finite set with the discrete topology.

A set $\Pi$ of partial permutations of a set $X$ determines a  
pseudogroup denoted $\H_\Pi$, namely the pseudogroup 
generated by all restrictions of the elements $p\in\Pi$. For example,
the pseudogroup associated to a Cameron
 permutoid $(\Pi_\rho; B_{2\rho})$ consists of all maps $U\to V$, with
$U,V\subseteq B_{2\rho}$, that are restrictions of left-muliplications
$x\mapsto g.x$ on $G$.

If $(\Pi;X)$ is a permutoid, then by passing
from $\Pi$ to $\H_\Pi$ one loses the crucial condition \ref{prob}(2).
Correspondingly, $\H_\Pi$ can always be
embedded in the pseudogroup $\H_{\Pi'}$ associated to a set $\Pi'$ of permutations
of $X$: take any choice of extension $p'\in {\rm{Perm}}(X)$ for $p\in\Pi$.

A more substantial analogue of Problem \ref{prob} in the context of pseudogroups
arises when one restricts attention to pseudogroups that
are {\em rigid} in the sense that maps are defined by their value at any point.

\begin{definition} A permutoid $(\Pi;X)$ is {\em rigid} if 
for all $p\neq q\in\Pi$, there is no $x\in X$ such that $p(x)= q(x)$.

A pseudogroup $\H$ is {\em rigid} if $f\cup g \in \H$
whenever $f,g\in \H$ and  $f(x)=g(x)$ for some $x\in X$ (equivalently, every
element of $\H$ has a unique maximal extension).
\end{definition}

A basic example of a rigid pseudogroup is the pseudogroup $G\ltimes X$
associated to a free action of a group on a space $X$ (in our case a finite
set with the discrete topology). The elements of this pseudogroup are the 
restrictions of the transformations in the action. In close analogy with Problem 
\ref{prob}, one would like to know, given a finite rigid pseudogroup, $\H$ on $X$,
if one can embed $X$ in a finite set $Y$ so that the elements of $\H$ are restrictions
of transformations of $Y$ in a free action of a finite group $G$; in other
words we wish to embed $\H$ in some $G\ltimes Y$. When this can be done, we say that
$\H$ is {\em developable}.

\begin{theorem}\label{t:ps} 
There does not exist an algorithm that can determine whether or not a finite, rigid
pseudogroup has a finite development.
\end{theorem}

The proof of this theorem is implicit in our earlier arguments; to translate
them we need the following lemma.

\begin{lemma} \label{l:trans}
Let $\H$ be a rigid pseudogroup on a finite set $X$ and let $\Lambda\subset \H$ be
the set of maximal elements.
\begin{enumerate}
\item $(\Lambda;X)$ is a rigid permutoid.
\item $\H_{\Lambda} = \H$.
\item If $\H=\H_{\Pi}$ for some permutoid $(\Pi;X)$,
then the map that assigns each $p\in\Pi$ to its maximal
extension in $\H$ defines an extension of permutoids $(\Pi;X)\to (\Lambda;X)$.
\item If $\H$ is developable, then so is $(\Lambda;X)$ (and hence $(\Pi;X)$).
\end{enumerate}
\end{lemma}

\begin{proof} The first three items follow easily from the
definitions. For example, if $p,q\in\Pi$ and $q(x)\in \ran(q)\cap\dom(p)$, then
the unique maximal element $r\in \H$ with $r(x)=pq(x)$ is the unique 
element of $\Pi$ such that $r$ extends $p\cdot q$. 

(4). If $G\ltimes Y$ is a finite development of $\H$, then each maximal
element $p\in \H$ is the restriction of the action of a unique $\hat{p}\in G$,
so  $\hat{r}=\hat{p}\hat{q}$ if $r$ extends $p\cdot q$. Thus 
$\Phi:X\hookrightarrow Y$ and 
 $\phi(p):=\hat{h}$ define a development of $(\Pi;X)$.
\end{proof}

Let $G$ be a group with finite generating set $A$ and let $B_r$ denote the ball of radius
$r$ about $1\in G$ in the corresponding word metric.
Earlier, we considered the permutoid $\pr=(\Pi_\rho; B_{2\rho})$. The 
associated pseudogroup $\H_{\Pi_\rho}$ consists of all maps $U\to V$, with
$U,V\subseteq B_{2\rho}$, that are restrictions of left multiplications
$\lambda_g:G\to G$.

\begin{proposition}\label{p2} Let $G$ be a group with finite presentation $\<A\mid R\>$, 
let $\rho>\frac 1 2 \max\{|r|:r\in R\}$ be an integer
and consider the permutoid $\pr =  (\Pi_\rho; B_{2\rho})$.
The following conditions are equivalent.
\begin{enumerate}
\item $G$ has a non-trivial finite quotient;
\item $\pr$ has a quotient permutoid
that is developable;
\item $\pr$ has a quotient permutoid
that  has a {\em rigid} finite development;
\item $\pr$ has a quotient permutoid whose associated pseudogroup is
rigid and developable.
\end{enumerate}
\end{proposition}

\begin{proof} We proved the equivalence of (1) and (2) in the proof of Proposition \ref{prop}.
(3) implies (2), trivially, and (1) implies (3) because if $Q$ is a finite
quotient of $G$, then the action of $Q$ by left-multiplication on itself provides
a rigid development for some quotient $P$  of $\pr$.  

Moreover, the rigid pseudogroup $Q\ltimes Q$  associated to this action (where $Q$ acts on itself by left multiplication)
is a development of the pseudogroup defined by $P$, and therefore (1) implies (4). 
Finally,  Lemma \ref{l:trans}(4) tells us that (4) implies (3). 
\end{proof}

\noindent{\em Proof of Theorem \ref{t:ps}.}  
We follow the proof of Proposition \ref{prop}. Taking finite presentations 
in a class $\mathfrak P$ 
where there is a uniform solution to the word problem but no algorithm that
can determine if the groups presented have finite quotients or not, we construct the
permutoids $\pr$ as above, list the finitely many quotients of each 
$\pr$, then pass to the pseudogroups defined by these quotients,
retaining only those pseudogroups
that are rigid (an easy check).
If there were an algorithm that could determine developability for rigid 
pseudogroups, then we would apply it to the members of the resulting list and thereby
(in the light of Proposition \ref{p2}) determine which of the groups 
with presentations in $\mathfrak P$  have finite quotients. This would
contradict our choice of $\mathfrak P$, and therefore no such algorithm exists.
\qed 
 
\begin{remark}
The undecidability phenomena that we have articulated in the language of 
permutoids and pseudogroups can equally be expressed in the language of 
groupoids or inverse semigroups (cf.~\cite{Lawson}). In the context of inverse semigroups,
Steinberg \cite{strongly0isundec} has proved a result similar to Theorem \ref{t:ps} {(see Theorem \ref{t:me}).  See Theorem \ref{t:undec.strongF} for a reformulation of Theorem \ref{t:ps} in the language of inverse semigroups.  (See also
\cite{hall+}.)}~Also, instead of 
considering finite sets, one could consider sets of partial automorphisms of 
simplicial complexes, for example.
\end{remark} 

\appendix
\section{Inverse semigroups}

\setcounter{footnote}{0}
\renewcommand*{\thefootnote}{\fnsymbol{footnote}}

\begin{center}{\sc Benjamin Steinberg}\footnote{{Department of Mathematics, City College of New York, Convent Avenue at 138th Street, New York, New York 10031, USA
\email{bsteinberg@ccny.cuny.edu}}}\footnote{{This work was partially supported by a grant from the Simons Foundation(\#245268
to Benjamin Steinberg), the Binational Science Foundation of Israel and the US (\#2012080 to Benjamin Steinberg), a PSC-CUNY grant and a CUNY Collaborative Incentive Research Grant.}}
\end{center}
\medskip

\renewcommand*{\thefootnote}{\arabic{footnote}}
\setcounter{footnote}{1}

The purpose of this appendix is to recast {the above results}~in the language of inverse semigroups, which is where {I believe}~they are most naturally stated, and to show how they fit into a body of literature already devoted to the subject.

Inverse semigroups were developed independently by Preston and Wagner to handle partial symmetry in much the way that groups deal with symmetry.  They are algebraic structures abstracting pseudogroups of partial homeomorphisms of a topological space.  Formally, an \emph{inverse semigroup} is a semigroup $S$ such that, for each $s\in S$, there exists a unique element $s^*\in S$ (called the \emph{inverse} of $s$) such that \[ss^*s=s\quad \text{and}\quad s^*ss^*=s^*.\]  A good reference on inverse semigroup theory is Lawson's book~\cite{Lawson}.  Note that $s\mapsto s^*$ is an involution of $S$ that satisfies the additional property that $ss^*tt^*=tt^*ss^*$ for all $s,t\in S$. Consequently, the set $E(S)$ of idempotents of $S$ is a commutative subsemigroup of $S$.

Every group is an inverse semigroup.  A fundamental example is the pseudogroup $I_X$ of all partial homeomorphisms of a topological space $X$. Formally, a partial homeomorphism of $X$ is a homeomorphism $f\colon U\to V$ between open subsets of $X$.  The composition of partial homeomorphisms is defined where it makes sense: if $f\colon U\to V$ and $g\colon W\to Z$ are partial homeomorphisms of $X$, then their product in $I_X$ is the composition \[f\circ g\colon g\inv(U\cap Z)\to f(U\cap Z).\]  If $f\colon U\to V$ is a partial homeomorphism of $X$, then $f^{*}$ is the inverse mapping $f^{-1}\colon V\to U$.  
{Bridson and Wilton term}~an inverse submonoid of $I_X$ closed under taking restrictions a \emph{pseudogroup}.  Note that, unlike~
{Bridson and Wilton},~we allow an empty partial homeomorphism, which is the zero element of $I_X$.  When $X$ is discrete, we call a partial homeomorphism of $X$ a \emph{partial permutation}. In this case, $I_X$ is called the \emph{symmetric inverse monoid} on the set $X$.  The Preston--Wagner theorem~\cite{Lawson} asserts that every inverse semigroup can be faithfully represented as a semigroup of partial permutations of its underlying set.

Inverse semigroups are closely connected to \'etale groupoids and have recently played some role in the theory of $C^*$-algebras, in part due to this connection. Also, inverse semigroups are precisely the $*$-semigroups of partial isometries of a Hilbert space. Many naturally arising operator algebras, like Cuntz--Krieger algebras, are generated by partial isometries.
 See, for instance, the book of Paterson~\cite{Paterson}, the long paper of Exel~\cite{Exel} and the papers of Nica~\cite{Nica} and of Khoshkam and Skandalis~\cite{Skandalis}.

One of the fundamental problems in inverse semigroup theory is that of extending partial permutations to permutations with some additional constraints.  This theme can be found in Lawson's book~\cite{Lawson}, as well as papers like~\cite{MSW,ASUnd,Herwig,strongly0isundec,geoams,Ash,ehud,Coulbextend}, to name but a few. Many of these papers connect extending partial permutations of a finite set to permutations of a bigger finite set, subject to constraints, to the profinite topology on appropriate groups.  The results {in the main body of this paper}~fit squarely into this body of work. 

There is a natural partial order on any inverse semigroup $S$, generalizing the restriction ordering on the pseudogroup $I_X$.  Namely, one puts $s\leq t$ if and only if $s=te$ for some idempotent $e\in E(S)$.  Moreover, the multiplication and the involution are order preserving. The product of two idempotents is their meet in this order.  To every inverse semigroup is associated a group $G(S)$, called its \emph{maximal group image} or \emph{group of germs}, defined by identifying two elements with a common lower bound.
  For instance, the  group of germs of the inverse monoid of isomorphisms between finite index subgroups of a group $G$ is called its abstract commensurator.

Let $\sigma\colon S\to G(S)$ be the canonical surjection.  Clearly, $E(S)\subseteq \sigma\inv(1)$.  One says that $S$ is \emph{$E$-unitary} if $\sigma\inv(1)=E(S)$ or, equivalently, if $s\geq e$ for some idempotent $e$, then $s$ is an idempotent.  If $S\subseteq I_X$ is an inverse subsemigroup  closed under taking non-empty restrictions, then $S$ is $E$-unitary if and only if each element of $S$ fixing a point of $X$ fixes its entire domain.

An \emph{$F$-inverse monoid} is an inverse monoid $M$ such that each element $m\in M$ is below a unique maximal element of $M$.  An $F$-inverse monoid is $E$-unitary since $1$ is the unique maximal element above any idempotent and only idempotents are below $1$. Let $\max\colon M\to M$ be the map sending an element to the unique maximal element above it.  Then $G(S)$ can be identified with $\max(M)$ equipped with the product $s\odot t=\max(st)$. The free inverse monoid is an important example of an $F$-inverse monoid~\cite{Lawson}.  

Let us consider examples from other areas of mathematics.
If $X$ is an irreducible algebraic variety, then the inverse monoid of isomorphisms between open subvarieties of $X$ is an $F$-inverse monoid and the group of birational automorphisms of $X$ can be identified with the group of germs of this monoid.  Examples of $F$-inverse monoids also occur in geometric group theory.  Birget~\cite{birgetthomp} observed that Thompson's group $V$ is the group of germs of the $F$-inverse monoid of all partial isomorphisms between finitely generated essential right ideals of a free monoid on $2$ letters. 


An inverse semigroup $S$ with a zero element cannot be $E$-unitary unless it consists only of idempotents because $s\geq 0$ for all $s\in S$.  Various attempts have been made to extend the notion to the setting of inverse semigroups with zero.  The weakest notion is that of an $E^*$-unitary inverse semigroup.  An inverse semigroup with zero  is \emph{$E^*$-unitary} if $s\geq e\neq 0$ with $e$ an idempotent implies that $s$ is an idempotent.  (Note that earlier papers also used the term ``$0$-$E$-unitary.'')  A pseudogroup $S\subseteq I_X$ is $E^*$-unitary if and only if each element of $S$ that fixes a point, fixes its entire domain. Equivalently, a pseudogroup is $E^*$-unitary if and only if whenever two elements agree at a point, they agree on the intersection of their domains.

The analogue of developability in inverse semigroup theory is the notion of a strongly $E^*$-unitary inverse semigroup, introduced independently in~\cite{FountEunit} and~\cite{Lawson0eunit}.   A good survey article on the subject is Lawson~\cite{Eunitsurvey}.   Let $S$ be an inverse semigroup with zero.  Then a \emph{partial homomorphism} from $S$ to a group $G$ is a mapping $\p\colon S\setminus \{0\}\to G$ such that $\p(st)=\p(s)\p(t)$ whenever $st\neq 0$.  
It is easily checked that $E(S)\setminus \{0\}\subseteq \pinv(1)$
 and, consequently, $\p$ is constant on connected components of the Hasse diagram of $S\setminus \{0\}$. 
One also checks that $\p(s^*)=\p(s)\inv$. We say that $\p$ is \emph{idempotent pure} if $\pinv(1)=E(S)\setminus \{0\}$.  An inverse semigroup with zero is  \emph{strongly $E^*$-unitary} if it admits an idempotent pure partial homomorphism to a group.

  Every inverse semigroup $S$ with zero has a \emph{universal group} $U(S)$, equipped with a partial homomorphism $\theta\colon S\setminus \{0\}\to U(S)$, such that every other partial homomorphism from $S$ to a group factors uniquely through $\theta$.  One can construct $U(S)$ as the group generated by $S\setminus \{0\}$ with relations that put the word $(s,t)$ equal to the symbol $(st)$ whenever $st\neq 0$ in $S$.  See~\cite{strongly0isundec,Eunitsurvey} for details.  It is easy to see that $S$ is strongly $E^*$-unitary if and only if the universal partial homomorphism $\theta$ is idempotent pure.  Strongly $E^*$-unitary inverse semigroups are clearly $E^*$-unitary since if $s\geq e\neq 0$ with $e\in E(S)$, then $\theta(s)=\theta(e)=1$ and hence $s\in E(S)$ because $\theta$ is idempotent pure.

 Alternatively, strongly $E^*$-unitary inverse semigroups can be described as Rees quotients of $E$-unitary inverse semigroups by ideals (cf.~\cite{strongly0isundec}).  An ideal $I$ in a semigroup $S$ is a non-empty subset such that $SI\subseteq I$ and $IS\subseteq I$.  The \emph{Rees quotient} $S/I$ is the quotient of $S$ by the congruence which identifies $I$ to a single element (which will be the zero of the quotient).

We say that $S$ is strongly $E^*$-unitary over a class $\mathcal C$ of groups if it admits an idempotent pure partial homomorphism to a group in $\mathcal C$.

Examples of strongly $E^*$-unitary inverse semigroups abound in $C^*$-algebra theory.  For instance, the graph inverse semigroups associated to Cuntz--Krieger $C^*$-algebras~\cite{Paterson} are strongly $E^*$-unitary, as are tiling semigroups~\cite{tiling} and Toeplitz inverse semigroups~\cite{Nica}.  The property of being strongly $E^*$-unitary for an inverse semigroup is closely connected to representability of its associated operator algebras as cross products or partial cross products of commutative $C^*$-algebras with groups.  See~\cite{Milan} for details.

The author's paper~\cite{strongly0isundec} provides a connection with geometric group theory.  Following Stallings~\cite{Stallings}, a morphism of graphs $\p\colon \Gamma\to \Gamma'$ is an \emph{immersion} if it is injective on each star.  Let $A$ be a set.  Then there is a bijection between immersions over the bouquet of $A$ circles and $A$-generated inverse semigroups of partial permutations, as was pointed out by Margolis and Meakin~\cite{MMimmersions}.   The point is that the monodromy action is only partially defined: each element of $A$ has at most one lift starting at any vertex under an immersion and the initial vertex is sent to the terminal vertex of the lift, when defined, by the monodromy action.

 In~\cite{strongly0isundec} the author shows that a connected graph immersion over a bouquet can be extended to a regular covering map if and only if the corresponding inverse semigroup generated by the monodromy action is strongly $E^{*}$-unitary and the stabilizer of each vertex consists of idempotents.  It was also shown that a finite connected graph immersion over a bouquet can be extended to a finite-sheeted regular covering if and only if the inverse semigroup generated by the monodromy action is strongly $E^{*}$-unitary over the class of finite groups and each vertex stabilizer consists of idempotents.

The author proved in~\cite{strongly0isundec} the following theorem.

\begin{theorem}[Steinberg]\label{t:me}
The following algorithmic problems are undecidable for finite $E^*$-unitary inverse semigroups.
\begin{enumerate}
\item  Determining whether $S$ is strongly $E^*$-unitary.
\item Determining whether  $S$ is strongly $E^*$-unitary over finite groups.
\item Determining whether $S$ is a Rees quotient of an $E$-unitary inverse semigroup.
\item Determining whether  $S$ is a Rees quotient of a finite $E$-unitary inverse semigroup.
\end{enumerate}
\end{theorem}

The proof of the first and third undecidability results use the undecidability of the word problem for groups, whereas the proof of the second and fourth use the undecidability of the uniform word problem for finite groups~\cite{slobod}.

The analogue of $F$-inverse monoids in the context of monoids with zero was first considered by Nica in his work on operator algebras~\cite{Nica} and also is featured in the work of Khoshkam and Skandalis~\cite{Skandalis}.   See also~\cite{Eunitsurvey}.   An inverse monoid $M$ with zero is an \emph{$F^*$-inverse monoid} if each non-zero element of $M$ is below a unique maximal element.  Since each non-zero idempotent of $M$ is below $1$, it follows that $F^*$-inverse monoids are $E^*$-unitary.  A pseudogroup is an $F^*$-inverse monoid if and only if it is a rigid pseudogroup in the sense {of Section \ref{sec: Rigid development}}.


The main results {above}~can be viewed as proving the analogues of the second and fourth
items of Theorem~\ref{t:me} in the context of $F^*$-inverse monoids instead of $E^*$-unitary inverse semigroups.   Since being $F^*$-inverse is a more restrictive condition, this makes the results {of Bridson and Wilton}~stronger than Theorem~\ref{t:me}.

In the context of the fourth item, we need a definition.
It is natural to call an inverse monoid $M$ with zero \emph{strongly $F^*$-inverse}\footnote{Lawson uses strongly $F^*$-inverse to mean the conjunction of $F^*$-inverse and strongly $E^*$-unitary in~\cite{Eunitsurvey}, but our notion seems more aptly named.} if there is a partial homomorphism $\p\colon M\setminus \{0\}\to G$ with $G$ a group such that each non-empty fiber of $\p$ has a maximum element.  In this case, $\p$ must be idempotent pure (since $1$ is maximal in its fiber) and $M$ is both $F^*$-inverse and strongly $E^*$-unitary.  The operator algebras associated to strongly $F^*$-inverse monoids in our sense will be strongly Morita equivalent to full cross products of a group with a commutative $C^*$-algebra~\cite{Skandalis,Milan}, a feature not necessarily enjoyed by algebras of strongly $E^*$-unitary inverse semigroups.
 If $\mathcal C$ is a class of groups, we say that $M$ is \emph{strongly $F^*$-inverse over $\mathcal C$} if it admits a partial homomorphism to a group in $\mathcal C$ such that each non-empty fiber has a maximum element.  The interest in strongly $F^*$-inverse monoids over the class of finite groups stems from  the following proposition.

\begin{proposition}\label{p:ideal.form}
Let $M$ be a finite inverse monoid with zero.  Then $M$ is strongly $F^*$-inverse over the class of finite groups if and only if $M$ is isomorphic to a Rees quotient of a finite $F$-inverse monoid by an ideal.
\end{proposition}
\begin{proof}
Suppose first that $M\cong S/I$ with $S$ a finite $F$-inverse monoid.  Let $\sigma\colon S\to G(S)$ be the maximal group image homomorphism.  Note that $G(S)$ is finite and that each fiber of $\sigma$ has a maximum element (cf.~\cite{Lawson}).   We can identify $M\setminus\{0\}$ with the complement $S\setminus I$ of $I$.  Then $\sigma|_{S\setminus I}$ is a partial homomorphism (cf.~\cite{strongly0isundec}) and each non-empty fiber clearly still has a maximum element. Thus $M$ is strongly $F^*$-inverse over finite groups.

For the converse, let us suppose that $M$ admits a partial homomorphism  $\p\colon M\setminus \{0\}\to G$ with $G$ a finite group such that each non-empty fiber contains a maximum element. Put \[S=(\{0\}\times G)\cup \{(m,\p(m))\mid m\in M\setminus \{0\}\}.\]  The reader easily verifies that $S$ is a submonoid of $M\times G$ and $I=\{0\}\times G$ is an ideal of $M$.   Moreover, trivially $M\cong S/I$.  Of course, $S$ is finite.  It remains to observe that $S$ is an $F$-inverse monoid. We shall use that the natural partial order on a group is equality, the natural partial order on product is the product order and that the natural partial order on a subsemigroup is induced by that of the ambient semigroup. First note that if $g\notin \p(M\setminus \{0\})$, then clearly $(0,g)$ is a maximal element of $S$.  If $g\in \p(M\setminus \{0\})$, let $\max(g)$ be the maximum element of $\pinv(g)$.  Then $(\max(g),g)$ is the unique maximal element of $S$ above $(0,g)$ for $g\in \p(M\setminus \{0\})$ and
$(\max(\p(m)),\p(m))$ is the unique maximal element  above $(m,\p(m))$ for $m\in M\setminus \{0\}$.  Thus $S$ is an $F$-inverse monoid.
\end{proof}

A minor modification of the proof shows that  strongly $F^*$-inverse monoids are precisely the Rees quotients of $F$-inverse monoids.



{Let us next verify that a rigid pseudogroup $S$ is developable in the sense of Section \ref{sec: Rigid development} if and only if it is strongly $E^*$-unitary over the class of finite groups.  We say the development $G\ltimes Y$ of $S\subseteq I_X$ is \emph{faithful} if distinct maximal elements of $S$ are extended by distinct elements of $G$.  Observe that if there is some element $x\in X$ common to the domain of all maximal elements of $S$, then the development is automatically faithful.  Indeed, if $s$ and $t$ are distinct maximal elements of $S$ extended by the same element $g\in G$, then $sx=gx=tx$.  But then $s|_{\{x\}}=t|_{\{x\}}$ has two common maximal upper bounds, $s$ and $t$, a contradiction.}

\begin{proposition}\label{p:translate}
Let $X$ be a finite set and $S\subseteq I_X$ a rigid pseudogroup.  Then $S$ is developable if and only if $S$ is strongly $E^*$-inverse over finite groups.  Moreover, $S$ is strongly $F^*$-inverse over finite groups if and only if $S$ has a faithful development.
\end{proposition}
\begin{proof}
Suppose first that $S$ is developable and let $Y\supseteq X$ be a finite set and $G$ a finite group acting freely on $Y$ such that each element of $S$ is a restriction of an element of $G$.  Because the action is free, for each $s\in S\setminus \{0\}$, there is a unique element $\p(s)\in G$ such that $s\leq \p(s)$ (i.e., $s$ is a restriction of $\p(s)$).  If $s,s'\in S$ with $ss'\neq 0$, then $ss'\leq \p(s)\p(s')$ and so $\p(ss')=\p(s)\p(s')$.  Thus $\p\colon S\setminus \{0\}\to G$ is a partial homomorphism. Suppose that $s\in \pinv(1)$.  Then $s$ is a restriction of the identity and hence an idempotent.  Thus $\p$ is idempotent pure and so we may conclude that $S$ is strongly $E^*$-unitary over finite groups. If the development is faithful, then $\p$ is injective on maximal elements of $S\setminus \{0\}$ and hence each non-empty fiber of $\p$ has a maximum element and so $S$ is strongly $F^*$-inverse over finite groups.

To prove the converse, assume that $S$ is strongly $E^*$-unitary over finite groups and let $\p\colon S\setminus \{0\}\to G$ be an idempotent pure partial homomorphism with $G$ a finite group.  Fix a transversal $T$ to the set of orbits of $S$ on $X$ and let $Y=G\times T$ with the action $g'(g,t)=(g'g,t)$.  Then $Y$ is a finite set acted upon freely by $G$.  Define $\psi\colon X\to Y$ by $\psi(st)=(\p(s),t)$ for $t\in T$ and $s\in S$ (using that $X=ST$). First we verify that $\psi$ is well defined.  If $st=s't'$ with $s,s'\in S$ and $t,t'\in T$, then we must have $t=t'$ and so $st=s't$.  Therefore, $s,s'$ have a common upper bound $\til s$ because $S$ is rigid (namely the maximal element above $s|_{\{t\}}=s'|_{\{t\}}$).  But then $\p(s)=\p(\til s)=\p(s')$ and so $\psi$ is well defined.  To see that $\psi$ is injective, suppose that $\psi(x)=\psi(y)$.  Write $x=st$ and $y=s't'$ with $s,s'\in S$ and $t,t'\in T$.  Then $(\p(s),t)=(\p(s'),t')$ and so $t=t'$, $\p(s)=\p(s')$.  Note that $s's^*x=s's^*st=s't=y$ and so $s's^*\neq 0$.  Therefore, we have $\p(s's^*)=\p(s')\p(s^*)=\p(s')\p(s)\inv=1$.  As $\p$ is idempotent pure, we conclude that $s's^*$ is idempotent and hence $y=s's^*x=x$ as an idempotent partial permutation fixes its domain.

It remains to verify that $\p(s')$ extends $s'$ for $s'\in S\setminus \{0\}$.  So let $x\in X$ belong to the domain of $s'$ and write $x=st$ with $s\in S$ and $t\in T$.  Then $s'x=s'st$ and so $s's\neq 0$.   Therefore, we have
\begin{eqnarray*}
\p(s')\psi(x)&=&\p(s')(\p(s),t) = (\p(s')\p(s),t)\\
&=&(\p(s's),t)=\psi(s'st)=\psi(s'x)
\end{eqnarray*}
and so the action of $\p(s')$ extends the action of $s'$ (after identifying $X$ with a subset of $Y$ via $\psi$).  If, in addition, each fiber of $\p$ has a maximum element, then $\p$ is injective on maximal elements of $S\setminus \{0\}$ and hence the development is faithful.
\end{proof}

Next we shall show that there is no real difference between studying rigid pseudogroups and arbitrary $F^*$-inverse monoids in our context.

\begin{lemma}\label{l:extend.partial}
Let $T$ be an $F^*$-inverse monoid and let $S$ be an inverse submonoid with zero of $T$ containing all the maximal elements of $T$.  Then each partial homomorphism $\p$ from $S$ to a group $G$ extends uniquely to $T$.  Moreover, if $\p$ is idempotent pure, then so is the extension and if $\p$ is injective on maximal elements of $S\setminus \{0\}$, then the same is true for the extension.
\end{lemma}
\begin{proof}
Let $\max\colon T\setminus \{0\}\to T\setminus \{0\}$ be the mapping sending a non-zero element of $T$ to the unique maximal element above it. By assumption, the image of $\max$ is contained in $S\setminus \{0\}$.  Thus if $\Phi\colon T\setminus \{0\}\to G$ is an extension of $\p$, we must have $\Phi(t)=\Phi(\max(t))=\p(\max(t))$ and so $\Phi$ is unique, if it exists. We must now show that $\Phi(t)=\p(\max(t))$ is a partial homomorphism.  Suppose that $t,t'\in T\setminus \{0\}$ with $tt'\neq 0$ and put $s=\max(t)$ and $s'=\max(t')$.  Then $0\neq tt'\leq ss'$ and so $\max(tt')=\max(ss')$.  Thus \[\Phi(t)\Phi(t')=\p(s)\p(s')=\p(ss')=\p(\max(ss'))=\Phi(tt')\] as required.

Suppose, in addition, that $\p$ is idempotent pure and $\Phi(t)=1$.  Then $\p(\max(t))=1$ and so $\max(t)$ is an idempotent.  Therefore, $t$ is an idempotent as the set of idempotents is an order ideal in an inverse semigroup. Thus $\Phi$ is idempotent pure.  Clearly, $\Phi$ separates the maximal non-zero elements if $\p$ does as $S$ and $T$ have the same maximal elements.
\end{proof}

Now we show that every $F^*$-inverse monoid embeds into a rigid pseudogroup satisfying the conditions of Lemma~\ref{l:extend.partial}.

\begin{proposition}\label{p:make.rigid}
Let $S$ be an $F^*$-inverse monoid.  Then there is a a rigid pseudogroup $T\subseteq I_{S\setminus \{0\}}$ such that $S$ is an inverse submonoid with zero of $T$ containing all the maximal elements.  In particular, $S$ is finite if and only if $T$ is finite.  Moreover, $S$ is strongly $E^*$-unitary (respectively, strongly $F^*$-inverse) over a class of groups $\mathcal C$ if and only if $T$ is strongly $E^*$-unitary (respectively, strongly $F^*$-inverse) over $\mathcal C$.
\end{proposition}
\begin{proof}
The inverse monoid $S$ acts faithfully on its underlying set via the Preston-Wagner representation $\rho\colon S\to I_S$ where $\rho(s)\colon s^*sS\to ss^*S$ is defined by $x\mapsto sx$~\cite{Lawson}. Clearly, $0$ is fixed by each element of $\rho(S)$.  Thus $S\setminus \{0\}$ is invariant under $\rho(S)$ and the restricted representation $\rho'\colon S\to I_{S\setminus \{0\}}$ is still faithful but has $\rho'(0)$ the empty map. Let $T$ be the pseudogroup of all restrictions of elements of $\rho'(S)$. Clearly, $\rho'(S)$ and $T$ have the same maximal elements by construction and each element of $T$ is below a maximal element of $\rho'(S)$ (as $S$ is an $F^*$-inverse monoid).  We claim that $T$ is $F^*$-inverse, i.e., a rigid pseudogroup.  To see this, we need to show that if $s,s'\in S$ are distinct maximal elements, then $\rho'(s)$ and $\rho'(s')$ have no common, non-empty restriction. Equivalently, we must show that $\rho'(s)$ and $\rho'(s')$ do not agree on any element common to their domains.  So suppose that $x\in S\setminus \{0\}$ belongs to the domain of both $\rho'(s)$ and $\rho'(s')$ and they agree on $x$.  Then $sx=s'x\neq 0$ and hence $sxx^*=s'xx^*\neq 0$ (since $s'xx^*x=s'x$).  As $xx^*$ is an idempotent, we have that $sxx^*=s'xx^*$ is a common non-zero lower bound of $s$ and $s'$ in $S\setminus \{0\}$.  This contradicts that $S$ is an $F^*$-inverse monoid.  We conclude that $T$ is a rigid pseudogroup.  From now on we identify $S$ with the submonoid $\rho'(S)$ of $T$.

Clearly, if $\p$ is partial homomorphism from $T$ to a group $G$ in $\mathcal C$ that is idempotent pure (respectively, has maximum elements in non-empty fibers) then the restriction of $\p$ to $S$ has the same property.  Thus if $T$ is strongly $E^*$-unitary (respectively, strongly $F^*$-inverse) over $\mathcal C$, then so is $S$.  Conversely, if $S$ admits an partial homomorphism to a group in $\mathcal C$ that is idempotent pure (respectively, has maximum elements in non-empty fibers), then so does $T$ by Lemma~\ref{l:extend.partial}.  This completes the proof.
\end{proof}

In light of Propositions~\ref{p:translate} and~\ref{p:make.rigid}, we can reformulate {Theorem \ref{t:ps}}~in the following equivalent way (contrast with Theorem~\ref{t:me}).

\begin{theorem}\label{t:undec.over.finite}
It is undecidable whether a finite $F^*$-inverse monoid is strongly $E^*$-unitary over finite groups.
\end{theorem}

The proof  of Theorem~{\ref{t:ps}}~uses rigid pseudogroups whose maximal elements are finite quotients of Cameron permutoids.  These rigid pseudogroups have the property that all maximal elements have the same domain.  Thus they are developable if and only if they have a faithful development by the remark preceding Proposition~\ref{p:translate}.
Therefore, Theorem~{\ref{t:ps}}~and Proposition~\ref{p:translate}  yield the following undecidability result.

\begin{theorem}\label{t:undec.strongF}
It is undecidable whether a finite $F^*$-inverse monoid is strongly $F^*$-inverse over finite groups.
\end{theorem}

Proposition~\ref{p:ideal.form} lets us reformulate Theorem~\ref{t:undec.strongF} in  a more appealing manner.

\begin{theorem}\label{t:undec.rees}
It is undecidable whether a finite $F^*$-inverse monoid is a Rees quotient of a finite $F$-inverse monoid by an ideal.
\end{theorem}

\bibliographystyle{plain}

\end{document}